\documentclass[11pt,reqno]{amsart}
\usepackage{amsmath,amsthm,amscd,amsfonts,amssymb,color}
\usepackage{color}
\usepackage{cite}
\usepackage[mathscr]{eucal}
\newtheorem{theorem}{Theorem}[section]

\newtheorem{corollary}[theorem]{Corollary}
\theoremstyle{definition}
\newtheorem{definition}[theorem]{Definition}
\newtheorem{example}[theorem]{Example}
\theoremstyle{remark}

\numberwithin{equation}{section}
\def\DJ{\leavevmode\setbox0=\hbox{D}\kern0pt\rlap
 {\kern.04em\raise.188\ht0\hbox{-}}D}

\begin{document}
\title[Existence of solutions to non-linear quadratic integral equations]{Existence of solutions to non-linear quadratic integral equations via measure of non-compactness}

\author[S. \ Karmakar]{Surajit \ Karmakar$^{1}$}
\address{$^1$Department of Mathematics,
\newline \indent National Institute of Technology  Durgapur,
\newline \indent India}
\email{surajit866@gmail.com}

\author[H. \ Garai]{Hiranmoy \ Garai$^{2}$}
\address{$^{2}$Department of Mathematics,
\newline \indent National Institute of Technology  Durgapur,
\newline \indent India}
\email{hiran.garai24@gmail.com}

\author[L.K. \ Dey]{Lakshmi Kanta \ Dey$^{3}$}
\address{$^{3}$Department of Mathematics,
\newline \indent National Institute of Technology  Durgapur,
\newline \indent India}
\email{lakshmikdey@yahoo.co.in}

\author[A. Chanda]{Ankush Chanda$^{4}$}
\address{$^{4}$Department of Mathematics,
\newline \indent Vellore Institute of Technology,  Vellore,
\newline \indent India}
\email{ankushchanda8@gmail.com}

\keywords{Measure of non-compactness, quadratic integral equations, Darbo fixed point theorem, Mizoguchi-Takahashi functions, Banach spaces. \\
\indent 2010 {\it Mathematics Subject Classification}.  $47$H$10$, $54$H$25$. \\
}

\begin{abstract}
The objective of this manuscript is to enquire for the solvability of a specific type of non-linear quadratic integral equations via the interesting notion of measure of non-compactness. Firstly, we inquire into couple of exciting fixed point theorems involving a measure of non-compactness in the setting of a Banach space. Subsequently, bringing into play a suitable measure of non-compactness and the acquired results, we discuss the existence of solution to the aforementioned kind of non-linear quadratic integral equations. 
\end{abstract}

\maketitle

\section{Introduction}
\baselineskip .55 cm
The study on integral equations has drawn a huge amount of attention from the enthusiasts as it plays a decisive role in spelling out
plenty of events and problems of real world. Utilizing several theories of functional analysis, topology and fixed point theory as the tools, the subject is growing fast with applications in applied mathematics, physics, engineering, economics, biological sciences and many other branches of science \cite{AOW,C14,OM,ZKKMRS}. Another crucial aspect of the study is the investigation of solvability conditions to integral equations as a whole, fixed point theory is one of the most vital and easiest of them. Mathematicians have explored a number of fixed point theorems and common fixed point theorems \cite{SB,AAM,AMR,GDC,KGDC} to guarantee the existence of a solution or more to a certain type of integral equations. In the existing literature, there are many a number of articles which deal with such scenarios via some suitable results from fixed point theory involving control functions while some others make use of findings in fixed point theory via measure of non-compactness. However, it is exciting to note that both of these approaches firstly enquire for existence of solutions and then look for solutions, and also are quite easy to handle. 

In his research article, Kuratowski \cite{K11} put forward the concept of a measure of non-compactness (in short, MNC) and this gave a new wing to the research in this direction. Afterwards, Darbo \cite{D4} employed this concept and obtained the fixed points of $\alpha$-set contraction defined on a closed, bounded and convex subset of a Banach space. This  result extends and generalizes the Schauder fixed point theorem remarkably, and is often brought into play as an essential tool to inspect the existence of a solution to a number of classes of non-linear equations. A few interesting works related to this notion can be revisited in \cite{AAM,AA,CL,CT,NAAH, ANCB} and the references therein.

On the other hand, the quadratic integral equations are frequently applied to the theory of radiative transfer, kinetic theory of gases, theory of neutron transport, biology and queuing theory, vehicular traffic theory and many more \cite{A2,B9,D2,HKZ,CZ}. Besides, this kind of integral equations has been a topic of extensive mathematical investigations by the researchers \cite{AOW,C14,ZKKMRS,V1,SB}.

Meanwhile, there comes an obvious question on whether some fixed point results involving both the control functions and measure of non-compactness can be employed to tackle such situations or not.
To meet this specific interest, in this article, we confirm a couple of fascinating fixed point results concerning a continuous operator defined on a non-empty, bounded, closed and convex subset of a Banach space. To add with, we come by some immediate corollaries from our conceived results. Alongside, we consider a certain class of aforementioned equations, non-linear quadratic integral equations (QIE) of Volterra type, as follows:
\begin{equation*}
x(t)=g(t,x(t))+\lambda \int_{0}^{t} \alpha_1(t,s)\zeta_1(s,x(s)) ds \int_{0}^{t} \alpha_2(t,s)\zeta_2(s,x(s)) ds,
\end{equation*}
for $t>0$, where $h,\mu_1,\mu_2,\zeta_1,\zeta_2$ are real-valued continuous functions defined on $\mathbb{R}^+ \times  \mathbb{R}$ and $\lambda$ is a positive constant. It may be noted that such kind of quadratic integral equations has a far-reaching applicability in more diversified fields. Further, the study of the solvability of such equations of Volterra type, employing the notion of measure of non-compactness is yet to appear in the literature. Therefore, we take up our obtained results to discuss some sufficient conditions for the existence of solution to a certain class of non-linear quadratic integral equations. Finally, our findings generalize, extend and compliment a number of results existing in the literature.
\section{Preliminaries}
This section deals with some essential notions, fundamental results and terminologies which are playing the lead roles in our findings. To begin with, we note down the definition of a measure of non-compactness.
\begin{definition}  \cite{BG}
A mapping $\sigma:\mathfrak{M}_E\to \mathbb{R}^+$ is said to be an MNC on $E$ if the following are satisfied:
\begin{enumerate}
\item[(i)]
the class $ker \sigma=\{A \in \mathfrak{M}_E:\sigma(A)=0\}$ is non-empty and $ker \sigma \subseteq \mathfrak{M}_E$;
\item[(ii)]
$A\subseteq B \implies \sigma(A)\leq \sigma(B)$;
\item[(iii)]
$\sigma(\overline{A})=\sigma(A)$;
\item[(iv)]
$\sigma(conv(A))=\sigma(A)$;
\item[(v)]
$\sigma(\lambda A+(1-\lambda)B)\leq \lambda \sigma(A)+(1-\lambda)\sigma(B)$;
\item[(vi)]
if $(A_n)$ is a sequence of closed sets from $\mathfrak{M}_E$ such that $A_{n+1}\subseteq A_n$ for $n \in \mathbb{N}$ and if $\displaystyle\lim_{n\to \infty}{\sigma(A_n)=0}$, then the set $A_\infty=\cap_{n=1}^{\infty}A_n$ is non-empty.
\end{enumerate}
\end{definition}
Here we note that the collection $ker \sigma$ is called as the kernel of the measure of non-compactness $\sigma$ and also recollect an essential property that $A_\infty \in ker \sigma.$ Additionally, using the result $\sigma(A_\infty)\leq \sigma(A_n)$ for all $n \in \mathbb{N}$, we can conclude that $\sigma(A_\infty)=0.$ Now, we put down the much acclaimed Schauder fixed point theorem.
\begin{theorem} \cite{BG}
Suppose that $A$ is a non-empty, bounded, closed and convex subset of any Banach space $M$. Then each continuous, compact mapping $\tau:A \to A$ owns at least one fixed point in $A$.
\end{theorem}
Hereafter, we make a note of one of the most exciting and remarkable generalizations of the previous theorem which is the Darbo fixed point theorem.
\begin{theorem} \cite{D4}
Let $A$ be a non-empty, bounded, closed and convex subset of a Banach space $M$ and suppose that $\tau:A \to A$ is  continuous. Further suppose that there exists a constant $k \in [0,1)$ such that
\[\sigma(\tau X)\leq k \sigma(X)\]
for a non-empty subset $X$ of $A$, where $\sigma$ is an arbitrary MNC defined in $M$. Then $\tau$ owns a fixed point in $A$.
\end{theorem}
The subsequent collection of functions was coined by Geraghty \cite{G} and the notion of such functions is pre-requisite for our conceived results.
\begin{definition} \cite{G}
Let $\Delta$ be the family of all functions $\alpha : \mathbb{R}^{+}\to [0,1)$ such that for any sequence $(t_n)$ whenever $\alpha(t_n)\to 1$ that implies $t_n \to 0$ as $n \to \infty$.
\end{definition}
Further, the concept of the following class of functions was presented by Altun and Turkoglu \cite{AT} and is employed to establish many results in this direction.
\begin{definition}\label{AT} \cite{AT}
Let $\mathfrak{F}([0,\infty))$ be the class of functions such that $\xi:[0,\infty)\to [0,\infty]$ and also let $\Theta$ be the the class of all operators 
\begin{eqnarray*}
\mathcal{O}(\bullet,.) : \mathfrak{F}([0,\infty))& \to &  \mathfrak{F}([0,\infty)),\\
 \xi & \to & \mathcal{O}(\xi,.)
\end{eqnarray*}
satisfying the following conditions:
\begin{enumerate}
\item[(i)] 
$\mathcal{O}(\xi;t)>0$ for all $t>0$ and $\mathcal{O}(\xi,0)=0$;
\item[(ii)]
$\mathcal{O}(\xi;t)\leq \mathcal{O}(\xi;s)$ for $t\leq s$;
\item[(iii)]
$\displaystyle \lim_{n\to\infty} \mathcal{O}(\xi;t_n)=\mathcal{O}(\xi;\displaystyle \lim_{n\to\infty} t_n)$;
\item[(iv)]
$\mathcal{O}(\xi;\max \{t,s\})=\max\{\mathcal{O}(\xi;t),\mathcal{O}(\xi;s)\}$ for some $\xi\in \mathfrak{F}([0,\infty))$.
\end{enumerate}
\end{definition}
\begin{example} \cite{AT}
Let $\xi:[0,\infty)\to [0,\infty)$ be a non-decreasing continuous function with $\xi(0)=0$ and $\xi(t)>0$ for $t>0$. Then 
\[\mathcal{O}(\xi;t)=\frac{\xi(t)}{1+\ln(1+\xi(t))}\]
satisfies all the aforementioned conditions.
\end{example}
The succeeding family of control functions was initially formulated by Nashine and Arab \cite{NA2} in their research article.
\begin{definition} \cite{NA2}
Let $\Psi$ denote the collection of functions $\eta: \mathbb{R}^{+} \to \mathbb{R}^{+}$ which satisfy the followings:
\begin{enumerate}
\item[(i)]
$\eta$ is non-decreasing;
\item[(ii)]
$\eta$ is continuous;
\item[(iii)]
$\eta^{-1}(\{0\})=\{0\}$.
\end{enumerate}
\end{definition}
The authors made use of the previously discussed functions and an arbitrary MNC to secure a couple of impressive fixed point results in \cite{NA2}. However, the ensuing class of functions is playing a vital cog in our manuscript. Arab et al. \cite{AMR} utilized this notion and enriched the literature with a few interesting fixed point results.
\begin{definition} \cite{AMR}
Let $\mathbb{F}$ be the class of functions $\mathcal{F}: [0,\infty)^2\to [0,\infty)$ which satisfy the following:
\begin{enumerate}
\item[(i)]
$\max\{x,y\}\leq \mathcal{F}(x,y)$ for all $x,y\geq 0$;
\item[(ii)]
$\mathcal{F}$ is continuous.
\end{enumerate}
\end{definition}
\begin{example}
Here we put down some of the examples which satisfy the previous definition.
\begin{enumerate}
\item[(i)]
$\mathcal{F}_1(x,y)=\max\{x,y\}$;
\item[(ii)]
$\mathcal{F}_2(x,y)=x+y$.
\end{enumerate}
\end{example}

\section{Main Results}
This section deals with some novel and exciting fixed point results involving the previously discussed notions. Here we present the very first one which is as follows.
\begin{theorem} \label{th1}
Let $A$ be a non-empty, bounded, closed and convex subset of a Banach space $E$ and $T:A \to A$ be a continuous operator satisfying
\begin{align}\label{1a}
\eta(\mathcal{O}(\xi;\mathcal{F}(\sigma(TY),\varphi(\sigma(TY)))))\leq \alpha(\mathcal{O}(\xi;\eta(\sigma(A))))\beta (\mathcal{O}(\xi;\mathcal{F}(\sigma(Y),\varphi(\sigma(Y)))))
\end{align}
for all $Y \subseteq A$, where $\alpha \in \Delta$, $\mathcal{F}\in \mathbb{F}$, $\eta \in \Psi$, $\beta:\mathbb{R}^{+}\to\mathbb{R}^{+}$ is a continuous function such that $\eta(t)>\beta(t)$ for all $t>0$ and $\varphi :\mathbb{R}^{+}\to\mathbb{R}^{+}$ is also a continuous function. Then $T$ has at least one fixed point in $A$.
\end{theorem}
\begin{proof}
Let us consider the sequence $(A_n)$ of subsets of $A$, where $A_0=A$ and $A_{n+1}=conv(TA_n)$ for $n\geq 0.$ Then we have,
 $TA_0=TA \subseteq A=A_0$, 
$A_1=conv (TA_0) \subseteq A \subset A_0$,
$A_2=conv(TA_1)\subseteq A_1\subset A_0$.
Thus continuing in this way we get
$$A_{n+1}\subseteq A_n\subseteq \cdots \subseteq A_0.$$
If $\sigma(A_{n_0})=0$ for some $n_0 \in \mathbb{N}$, then $A_{n_0} \in \mathfrak{N}_A$. Hence $A_{n_0}$ is a relatively compact subset of $A$. Since $T(A_{n_0})\subseteq conv(TA_{n_0})\subseteq A_{n_0}$, then by Schauder fixed point theorem $T$ has a fixed point. Therefore we can consider $\sigma(A_n)>0$ for $n\geq 0$.

Now from \eqref{1a}, we have,
\begin{align*}
\eta(\mathcal{O}(\xi;\mathcal{F}(\sigma(A_{n+1}),\varphi(\sigma(A_{n+1})))))=&\eta (\mathcal{O}(\xi;\mathcal{F}(\sigma(conv (TA_{n})),\varphi(\sigma(conv(TA_{n}))))))\nonumber\\
=&\eta (\mathcal{O}(\xi;\mathcal{F}(\sigma(TA_{n}),\varphi(\sigma(TA_{n})))))\nonumber\\
\leq & \alpha(\mathcal{O}(\xi;\eta(\sigma(A_n))))\beta(\mathcal{O}(\xi;\mathcal{F}(\sigma(A_{n}),\varphi(\sigma(A_{n})))))\nonumber\\
<&\beta(\mathcal{O}(\xi;\mathcal{F}(\sigma(A_{n}),\varphi(\sigma(A_{n})))))\nonumber\\
<&\eta(\mathcal{O}(\xi;\mathcal{F}(\sigma(A_{n}),\varphi(\sigma(A_{n}))))).
\end{align*}
As $\sigma(A_n)>0$ and $\eta(t)>\beta(t)$ when $t>0$ and also by the properties of $\eta$ and $\mathcal{O}(\bullet,.)$, it follows that the sequence $(\mathcal{F}(\sigma(A_n), \varphi(\sigma(A_n))))$ is non-increasing sequence of positive real numbers. Hence there exists a real number $\delta>0$ such that
\begin{equation}\label{e1}
\displaystyle \lim_{n\to \infty}\mathcal{F}(\sigma(A_{n+1}),\varphi(\sigma(A_{n+1})))=\displaystyle \lim_{n\to \infty}\mathcal{F}(\sigma(A_{n}),\varphi(\sigma(A_{n})))=\delta.
\end{equation}
\textbf{Case-I:}
Suppose $\delta=0$. Then, $$\displaystyle \lim_{n\to \infty}\mathcal{F}(\sigma(A_{n}),\varphi(\sigma(A_{n})))=0.$$
Since, $\mathcal{F}$ is continuous and $\mathcal{F}\in \mathbb{F}$, we have
\begin{align*}
\displaystyle \lim_{n\to \infty} \sigma(A_n)+\displaystyle \lim_{n\to \infty} \varphi (\sigma(A_n))\leq \displaystyle \lim_{n\to \infty} \mathcal{F}(\sigma(A_{n}),\varphi(\sigma(A_{n})))=& 0,\\
\implies \displaystyle \lim_{n\to \infty} \sigma(A_n)=&0.
\end{align*}
\textbf{Case-II:} Suppose $\delta>0$. Then we have,
\begin{align}\label{e2}
\eta(\mathcal{O}(\xi;\mathcal{F}(\sigma(A_{n+1}),\varphi(\sigma(A_{n+1})))))\leq \alpha(\mathcal{O}(xi;\eta(\sigma(A_n))))\beta(\mathcal{O}(\xi;\mathcal{F}(\sigma(A_{n}),\varphi(\sigma(A_{n}))))).
\end{align}
Letting $n\to \infty$ in \eqref{e2} and using \eqref{e1} we get,
\begin{equation*}
\eta (\mathcal{O}(\xi;\delta))\leq \displaystyle \lim_{n\to \infty} \alpha(\mathcal{O}(\xi;\eta(\sigma(A_n)))) \beta(\mathcal{O}(\xi;\delta)).
\end{equation*}
Since for all $t>0$, we have $\alpha<1$ and $\eta(t)>\beta(t)$, therefore,
\begin{align*}
\eta(\delta)\leq & \eta(\delta) \displaystyle \lim_{n\to \infty} \alpha(\mathcal{O}(\xi;\eta(\sigma(A_n))))\\
\Rightarrow 0\leq & \eta(\delta)\left\{\displaystyle \lim_{n\to \infty}\alpha(\mathcal{O}(\xi;\eta(\sigma(A_n))))-1\right\}.
\end{align*}
Therefore,
$$\displaystyle \lim_{n\to \infty}\alpha(\mathcal{O}(\xi;\eta(\sigma(A_n))))=1.$$
This implies that, 
\begin{align*}
\displaystyle \lim_{n\to \infty}\mathcal{O}(\xi;\eta(\sigma(A_n)))=&0\\
\Rightarrow \mathcal{O}(\xi;\eta(\displaystyle \lim_{n\to \infty}\sigma(A_n)))=& 0\\
\Rightarrow \eta(\displaystyle \lim_{n\to \infty}\sigma(A_n))=& 0\\
\Rightarrow \displaystyle \lim_{n\to \infty}\sigma(A_n)=& 0.
\end{align*}
Thus in both the cases we have,
\[ \displaystyle \lim_{n\to \infty}\sigma (A_n)= 0.\]
Since $(A_{n})$ is a decreasing sequence of subsets, i.e., $A_{n+1} \subseteq A_n $ for all $n \in \mathbb{N}$, we can claim that $A_{ \infty}= \displaystyle \cap_{n=1}^{\infty} A_n$ is a non-empty, closed and convex subset of $A$. Also we have $A_{\infty}$ is an element of $ker \sigma$. Therefore $A_{\infty}$ is compact and invariant under the mapping $T$. It follows from the Schauder fixed point theorem that $T$ has a fixed point in $A$.
\end{proof}
The subsequent corollary can be readily acquired from the above theorem.
\begin{corollary} \label{cor1}
Let $A$ be a non-empty, bounded, closed and convex subset of a Banach space $E$ and $T:A \to A$ be a continuous operator satisfying
$$\eta(\mathcal{O}(\xi;\sigma(TX)+\varphi(\sigma(TX))))\leq \alpha(\mathcal{O}(\xi;\eta(\sigma(X))))\beta (\mathcal{O}(\xi;\sigma(X)+\varphi(\sigma(X))))$$
for all $Y \subseteq A$, where $\alpha \in \Delta$, $\eta \in \Psi$, $\beta:\mathbb{R}^{+}\to\mathbb{R}^{+}$ is a continuous function with $\eta(t)>\beta(t)$ for all $t>0$ and $\varphi :\mathbb{R}^{+}\to\mathbb{R}^{+}$ is also a continuous function. Then $T$ has at least one fixed point in $A$.
\end{corollary}
\begin{proof}
Let us consider the map $\mathcal{F} \in \mathbb{F}$ defined by $\mathcal{F}(x,y)=x+y$ for all $x,y \in [0,\infty)$. The proof of this corollary follows from Theorem \ref{th1}.
\end{proof}
If we take $\varphi \equiv 0$ in Corollary \ref{cor1}, then we have the following corollary.
\begin{corollary}
Let $A$ be a non-empty, bounded, closed and convex subset of a Banach space $E$ and $T:A \to A$ be a continuous operator satisfying
$$\eta(\mathcal{O}(\xi;\sigma(TY)))\leq \alpha(\mathcal{O}(\xi;\eta(\sigma(Y))))\beta (\mathcal{O}(\xi;\sigma(Y)))$$
for all $Y \subseteq A$, where $\alpha \in \Delta$, $\mathcal{F}\in \mathbb{F}$, $\eta \in Eta$, $\beta:\mathbb{R}^{+}\to\mathbb{R}^{+}$ is a continuous function with $\eta(t)>\beta(t)$ for all $t>0$. Then $T$ has at least one fixed point in $A$.
\end{corollary}
The subsequent result is related with Mizoguchi-Takahashi functions and also $\Omega$-type $\sigma$-set contractions. Further, it can be noted that if $\chi:[0,\infty)\to [0,1)$ is a non-decreasing function or a non-increasing function, then $\chi$ is an Mizoguchi-Takahashi function (in short, MT-function).
\begin{definition} \cite{MT}
Suppose that $\Upsilon$ denotes the class of Mizoguchi-Takahashi functions which satisfy the Mizoguchi-Takahasi condition,
\[\limsup_{s\to t^+}\chi(s)<1\]
for all $t\in [0,\infty).$
\end{definition}
This family consists of many a number of such functions. However in the following, we talk over an interesting collection of control functions which is required for our succeeding result.
\begin{definition} \cite{NA2}
Let $\Omega$ denote the class of all functions $\omega:[0,\infty) \to [0,\infty)$ satisfying:
\begin{enumerate}
\item[(i)]
$\omega$ is non-decreasing;
\item[(ii)]
$\omega(t)=0 \Leftrightarrow t=0.$
\end{enumerate}
\end{definition}
Now, we are all set to illustrate another fixed point result taking into account the aforementioned notions.
\begin{theorem}
Let $A$ be a non-empty, bounded, closed and convex subset of a Banach space $E$ and $T:A \to A$ be a continuous operator satisfying;
$$\omega(\mathcal{O}(\xi;\mathcal{F}(\sigma(TY),\varphi(\sigma(TY)))))\leq \chi(\mathcal{O}(\xi;\omega(\sigma(Y))))\omega(\mathcal{O}(\xi;\mathcal{F}(\sigma(Y),\varphi(\sigma(Y)))))$$
for any $\Phi \neq Y \subseteq A$, where $\sigma$ is an arbitrary MNC and $\mathcal{O}(\bullet,.)\in \Theta$, $\varphi : \mathbb{R}^{+} \to \mathbb{R}^{+}$ is a continuous function, $\chi \in \Upsilon$ and $\omega \in \Omega$. Then $T$ has at least one fixed point in $A$.
\begin{proof}
Let us consider the sequence $(A_n)$ by $A_0=A$ and $A_{n+1}=conv(TA_n)$, for $n\geq 0$. If $\mathcal{F}(\sigma(A_n),\varphi(\sigma(A_n)))=0$ for some natural number $n=n_0$, then as $\mathcal{F} \in \mathbb{F}$, we have
\begin{align*}
\sigma(A_{n_0})+\varphi(\sigma(A_{n_0}))\leq \mathcal{F}(\sigma(A_{n_0}),\varphi(\sigma(A_{n_0})))=& 0\\
\Rightarrow \sigma(A_{n_0})=& 0,
\end{align*}
as $\sigma$ and $\varphi$ are non-negative functions. Hence $A_{n_0}$ is a compact set and $T(A_{n_0})\subseteq conv (TA_{n_0})=A_{n_0+1}\subseteq A_{n_0}$. Hence, by Schauder fixed point theorem, $T$ has a fixed point in $A$.

Now, we consider the case when
$$\mathcal{F}(\sigma(A_n),\varphi(\sigma(A_n)))>0$$ for all $n\geq 0$. Then we have, 
\begin{align*}
\omega(\mathcal{O}(\xi;\mathcal{F}(\sigma(A_{n+1}),\varphi(\sigma(A_{n+1})))))\leq & \omega(\mathcal{O}(\xi;\mathcal{F}(\sigma(conv(TA_{n})),\varphi(\sigma(conv(TA_{n}))))))\\
=& \omega(\mathcal{O}(\xi;\mathcal{F}(\sigma(TA_{n}),\varphi(\sigma(TA_{n}))))))\\
\leq & \chi (\mathcal{O}(\xi;\omega(\sigma(A_n))))\omega(\mathcal{O}(\xi;\mathcal{F}(\sigma(A_n),\varphi(\sigma(A_n))))).
\end{align*}
Since $\chi<1$, we get
$$\omega(\mathcal{O}(\xi;\mathcal{F}(\sigma(A_{n+1}),\varphi(\sigma(A_{n+1})))))\leq \omega(\mathcal{O}(\xi;\mathcal{F}(\sigma(A_n),\varphi(\sigma(A_n))))).$$
This shows that the sequence $(\omega(\mathcal{O}(\xi;\mathcal{F}(\sigma(A_n),\varphi(\sigma(A_n))))))$ is non-decreasing and bounded below. Therefore there exists a real number $\rho\geq 0$ such that
$$\displaystyle \lim_{n\to \infty}\omega(\mathcal{O}(\xi;\mathcal{F}(\sigma(A_n),\varphi(\sigma(A_n)))))=\rho.$$
We claim that $\rho=0$. In contrary, if possible, let $\rho>0$. Since $\chi \in \Upsilon$, we have $$\displaystyle \lim_{t \to \rho^{+}} \chi(\mathcal{O}(\xi;t))<1$$ and $$\chi(\mathcal{O}(\xi;\rho))<1.$$
Then there exists $\lambda \in [0,1)$ and $\varepsilon>0$ such that $\chi(\mathcal{O}(\xi;t))\leq \lambda$ for all $t \in [\rho,\rho+\varepsilon)$. Hence there exists a natural number $n_0 \in \mathbb{N}$ such that
$$\rho\leq \omega(\mathcal{O}(\xi;\mathcal{F}(\sigma(A_n),\varphi(\sigma(A_n)))))\leq \rho+\varepsilon$$ for all $n\geq n_0$. Now we have,
\begin{align*}
\omega(\mathcal{O}(\xi;\mathcal{F}(\sigma(A_{n+1}),\varphi(\sigma(A_{n+1}))))\leq & \chi (\mathcal{O}(\xi;\omega(\sigma(A_n)))))\omega(\mathcal{O}(\xi;\mathcal{F}(\sigma(A_n),\varphi(\sigma(A_n)))))\\
\leq & \lambda \omega(\mathcal{O}(\xi;\mathcal{F}(\sigma(A_n),\varphi(\sigma(A_n))))).
\end{align*}
Hence,
\begin{align}\label{e21}
\omega(\mathcal{O}(\xi;\mathcal{F}(\sigma(A_{n+1}),\varphi(\sigma(A_{n+1})))))\leq & \lambda \omega(\mathcal{O}(\xi;\mathcal{F}(\sigma(A_n),\varphi(\sigma(A_n))))).
\end{align}
Taking limit $n \to \infty$ in \eqref{e21} and applying the property of $\mathcal{O}(\bullet,.)$ we get,
$$\rho\leq \lambda \rho.$$
Since $\lambda \in [0,1)$, we have $$\rho=0.$$
This implies that
$$\displaystyle \lim_{n \to \infty}\omega(\mathcal{O}(\xi;\mathcal{F}(\sigma(A_n),\varphi(\sigma(A_n)))))=0.$$
Since $(\omega(\mathcal{O}(\xi;\mathcal{F}(\sigma(A_n),\varphi(\sigma(A_n))))))$ is non-increasing and $\omega$ is a non-decreasing sequence, by the property of $\mathcal{O}(\bullet,.)$, we conclude that $(\mathcal{F}(\sigma(A_n),\varphi(\sigma(A_n))))$ is a decreasing sequence of positive reals. So there is a $\delta'>0$ such that
$$\displaystyle \lim_{n \to \infty} \mathcal{F}(\sigma(A_n),\varphi(\sigma(A_n))) =\delta'.$$
Therefore, $$\mathcal{F}(\sigma(A_n),\varphi(\sigma(A_n))) \geq \delta'$$ for all $n \in \mathbb{N}$. Hence by the property of $\mathcal{O}(\bullet,.)$ we get
$$\mathcal{O}(\xi;\mathcal{F}(\sigma(A_n),\varphi(\sigma(A_n))))\geq \mathcal{O}(\xi; \delta').$$
As $\omega$ is non-decreasing, we obtain
$$\omega(\mathcal{O}(\xi;\mathcal{F}(\sigma(A_n),\varphi(\sigma(A_n)))))\geq \omega( \mathcal{O}(\xi; \delta')).$$
Now, letting $n \to \infty$ in the previous inequality, we get
$$0\geq \omega( \mathcal{O}(\xi; \delta')).$$
From the above it follows that, $\mathcal{O}(\xi; \delta')=0$. Therefore,
$$\displaystyle \lim_{n\to \infty}\mathcal{F}(\sigma(A_n),\varphi(\sigma(A_n)))=0.$$
Since, $\mathcal{F}\in \mathbb{F}$, we have $$\displaystyle \lim_{n\to \infty} \sigma(A_n)+ \displaystyle \lim_{n\to \infty} \varphi(\sigma(A_n))\leq \displaystyle \lim_{n\to \infty}\mathcal{F}(\sigma(A_n),\varphi(\sigma(A_n)))=0.$$

Since $(A_{n})$ is a decreasing sequence of subsets, i.e., $A_{n+1} \subseteq A_n$ for all $n \in \mathbb{N}$, we can claim that $A_{ \infty}= \displaystyle \cap_{n=1}^{\infty} A_n$ is a non-empty, closed and convex subset of $A$. Also, we have $A_{\infty}$ is an element of $ker \sigma$. Therefore $A_{\infty}$ is compact and invariant under the mapping $T$. It follows from the Schauder fixed point theorem that $T$ has a fixed point in $A$.
\end{proof}
\end{theorem}
\section{An Application}
In this section, we employ our results obtained in the previous section, to find some sufficient conditions for the existence of solution(s) of the following non-linear quadratic integral equation 
\begin{equation}\label{eq1}
x(t)=g(t,x(t))+\lambda \int_{0}^{t} \mu_1(t,s)\zeta_1(s,x(s)) ds \int_{0}^{t} \mu_2(t,s)\zeta_2(s,x(s)) ds,
\end{equation}
for $t>0$, where $g,\mu_1,\mu_2,\zeta_1,\zeta_2$ are real valued continuous functions defined on $\mathbb{R}^+ \times  \mathbb{R}$ and $\lambda$ is a positive constant.

We use the notation $E$ to denote the set of all real valued continuous bounded functions which are defined on $(0,\infty)$. Then we know that $E$ is a Banach space with respect to the sup norm. We now recall a special type MNC on $E$. Let $A$ be a non-empty subset of $E$ and $L$ be a positive real number. For $x\in A$ and $\varepsilon>0$, we use the notation $w^L(x,\varepsilon)$ to denote the set $\sup \{|x(t)-x(u)| : t,u \in [0,L], |t-u| \leq \varepsilon\}$. Next we use the following notations 
\begin{align*}
&w^L(A;\varepsilon)=\sup \{w^L(x,\varepsilon):x\in A\}\\
&w_0^L(TA)=\lim_{\varepsilon\to 0} w^L(TA;\varepsilon)\\
&w_0(TA)=\lim_{L\to \infty} w_0^L(TA).
\end{align*}

Again we use the notations 
\begin{align*}
&A(t)=\{x(t):x\in A\}\\
&diam ~A(t)=\sup\left\{|x(t)-y(t)|:x,y\in A\right\}\\
&\alpha(A)=\limsup_{t\to \infty} diam ~A(t).
\end{align*}
Now  we define a real valued function $\sigma$ on $\mathfrak{M}_E$ by $$\sigma(A)=w_0(A)+\alpha(A).$$ Then is can be checked that $\sigma$ is an MNC on $E$, see \cite {BG, B10}. 

In order to present some sufficient conditions for the existence of the solution of the integral equation given by Equation \ref{eq1}, we need to restrict the class of operators $\mathcal{O}(\xi,t)$ discussed in Definition \ref{AT}, by adding some mild assumption, which is as follows: $$\mathcal{O}(\alpha \xi;t)\leq \alpha \mathcal{O}( \xi;t)$$ for all $t>0$, where $\alpha\in (0,1).$

Next, we prove the following theorem.
\begin{theorem}
Let us consider the Equation \ref{eq1} and  assume that the following conditions hold
\begin{enumerate}
\item [$(i)$] there exists a constant $\gamma$ with $0<\gamma<1$ such that $|g(t,x)-g(t,y)| \leq \gamma |x-y|$ holds for all $x,y \in \mathbb{R}$ and $t\geq 0$;
\item [$(ii)$] $\displaystyle \lim_{t\to \infty} \left|\int_{0}^{t} [\mu_1(t,s)\zeta_1(s,x(s))-\mu_1(t,s)\zeta_1(s,y(s))]ds\right|=0$ and $$\lim_{t\to \infty} \left|\int_{0}^{t} [\mu_2(t,s)\zeta_2(s,x(s))-\mu_2(t,s)\zeta_2(s,y(s))]ds\right|=0$$ uniformly with respect to $x,y\in E$;
\item [$(iii)$] there exist two real numbers $A_1,A_2>0$ such that $|\int_{0}^{t}\mu_1(t,s)\zeta_1(s,x(s)) ds|\leq A_1$ and $|\int_{0}^{t}\mu_2(t,s)\zeta_2(s,x(s)) ds|\leq A_2$ for all $t>0$ and for all $x,y\in E$.
\end{enumerate}
Then the integral equation given by Equation \ref{eq1}, has a solution.
\end{theorem}
\begin{proof}
Let $E$ be the set of all continuous bounded functions which are defined on $(0,\infty)$. Then we know that $E$ is a Banach space with respect to the sup norm. Note that if we choose $x,y\in E$, then $$g(t,x(t))+\lambda \int_{0}^{t} \mu_1(t,s)\zeta_1(s,x(s)) ds \int_{0}^{t} \mu_2(t,s)\zeta_2(s,x(s)) ds $$ is also continuous on $(0,\infty)$, i.e., $$g(t,x(t))+\lambda \int_{0}^{t} \mu_1(t,s)\zeta_1(s,x(s)) ds \int_{0}^{t} \mu_2(t,s)\zeta_2(s,x(s)) ds \in E.$$ So the mapping $T$ defined on $E$ by $$Tx(t)=g(t,x(t))+\lambda \int_{0}^{t} \mu_1(t,s)\zeta_1(s,x(s)) ds \int_{0}^{t} \mu_2(t,s)\zeta_2(s,x(s)) ds $$ is indeed a self-mapping on $E$. It is an easy task to note that $x(t)$ is a solution of Equation of \ref{eq1} if and only if $x(t)$ is a fixed point of $T$. 

Let $A$ be an arbitrary non-empty subset of $E$. Then we know that
\begin{align*}
\mathcal{O}(\xi;\sigma(TA))&=\mathcal{O}(\xi;w_0(TA)+\alpha(TA)), ~\mbox{where}
%&\leq \mathcal{O}(\xi;w_0(TA)+ \mathcal{O}(\xi;\alpha(TA)).
\end{align*}
\begin{align*}
&w_0(TA)=\lim_{L\to \infty} w_0^L(TA)=\lim_{L\to \infty} \lim_{\varepsilon\to 0} w^L(TA;\varepsilon)\\
&w^L(TA;\varepsilon)=\sup \{w^L(x,\varepsilon):x\in TA\}\\
&w^L(x,\varepsilon)=\sup \{|x(t)-x(u)| : t,u \in [0,L], |t-u| \leq \varepsilon\},
\end{align*}
and
\begin{align*}
\alpha(TA)=\limsup_{t\to \infty} diam ~TA(t)
%&= \limsup_{t\to \infty} \mathcal{O}\Big(\xi; diam ~TA(t)\Big),
\end{align*}
where $$diam ~TA(t)=\sup\left\{|Tx(t)-Ty(t)|:x,y\in A\right\}.$$
%Here \begin{align*}
%O(f;w_0(TA))&=O(f;\lim_{L\to \infty} w_0^L(TX))\\
%&=\lim_{L\to \infty} O(f; w_0^L(TX))\\
%&=\lim_{L\to \infty} O(f; \lim_{\varepsilon\to 0} w^L(TX;\varepsilon))\\
%&= \lim_{L\to \infty} \lim_{\varepsilon\to 0}O(f; w^L(TX;\varepsilon)).
%\end{align*}
Now let $t,u\in [0,L]$ be such that $|t-u|\leq \varepsilon$. Without loss of generality, we assume that $t\geq u$. Then we have
\begin{align*}
&|Tx(t)-Tx(u)|\\
&=\Big|g(t,x(t))+\lambda \int_{0}^{t} \mu_1(t,s)\zeta_1(s,x(s)) ds \int_{0}^{t} \mu_2(t,s)\zeta_2(s,x(s)) ds-g(t,x(u))\\ &+\lambda \int_{0}^{u} \mu_1(u,s)\zeta_1(s,x(s)) ds \int_{0}^{u} \mu_2(u,s)\zeta_2(s,x(s)) ds \Big|\\
&\leq |g(t,x(t))-g(t,x(u))| +\lambda \Big|\int_{0}^{t} \mu_1(t,s)\zeta_1(s,x(s)) ds \int_{0}^{t} \mu_2(t,s)\zeta_2(s,x(s)) ds\\ &-\int_{0}^{u} \mu_1(u,s)\zeta_1(s,x(s)) ds \int_{0}^{u} \mu_2(u,s)\zeta_2(s,x(s)) ds\Big|\\
&\leq |g(t,x(t))-g(t,x(u))|+\lambda \Big|\int_{0}^{t} \mu_1(t,s)\zeta_1(s,x(s)) ds \int_{0}^{t} \mu_2(t,s)\zeta_2(s,x(s)) ds\\ & -\int_{0}^{t} \mu_1(t,s)\zeta_1(s,x(s)) ds\int_{0}^{t} \mu_2(u,s)\zeta_2(s,x(s)) ds\Big | \\ & +\lambda \Big|\int_{0}^{t} \mu_1(t,s)\zeta_1(s,x(s)) ds\int_{0}^{t} \mu_2(u,s)\zeta_2(s,x(s)) ds\\&- \int_{0}^{t} \mu_1(u,s)\zeta_1(s,x(s)) ds \int_{0}^{t} \mu_2(u,s)\zeta_2(s,x(s)) ds\Big|\\ &  +\lambda \Big|\int_{0}^{t} \mu_1(u,s)\zeta_1(s,x(s)) ds \int_{0}^{t} \mu_2(u,s)\zeta_2(s,x(s)) ds \\ & - \int_{0}^{u} \mu_1(u,s)\zeta_1(s,x(s)) ds \int_{0}^{u} \mu_2(u,s)\zeta_2(s,x(s)) ds\Big|\end{align*}
\begin{align*}
& = |g(t,x(t))-g(t,x(u))|+\lambda \Big|\int_{0}^{t} \mu_1(t,s)\zeta_1(s,x(s)) ds \int_{0}^{t} \mu_2(t,s)\zeta_2(s,x(s)) ds\\ & -\int_{0}^{t} \mu_1(t,s)\zeta_1(s,x(s)) ds\int_{0}^{t} \mu_2(u,s)\zeta_2(s,x(s)) ds\Big | \\ & + \lambda \Big|\int_{0}^{t} \mu_1(t,s)\zeta_1(s,x(s)) ds\int_{0}^{t} \mu_2(u,s)\zeta_2(s,x(s)) ds\\&- \int_{0}^{t} \mu_1(u,s)\zeta_1(s,x(s)) ds \int_{0}^{t} \mu_2(u,s)\zeta_2(s,x(s)) ds\Big|\\ & +\lambda \Big|\int_{0}^{t} \mu_1(u,s)\zeta_1(s,x(s)) ds \int_{0}^{u} \mu_2(u,s)\zeta_2(s,x(s)) ds \\&+ \int_{0}^{t} \mu_1(u,s)\zeta_1(s,x(s)) ds \int_{u}^{t} \mu_2(u,s)\zeta_2(s,x(s)) ds\\ & - \int_{0}^{u} \mu_1(u,s)\zeta_1(s,x(s)) ds \int_{0}^{u} \mu_2(u,s)\zeta_2(s,x(s)) ds\Big|\\
& = |g(t,x(t))-g(t,x(u))|+\lambda \Big|\int_{0}^{t} \mu_1(t,s)\zeta_1(s,x(s)) ds \int_{0}^{t} \mu_2(t,s)\zeta_2(s,x(s)) ds\\ & -\int_{0}^{t} \mu_1(t,s)\zeta_1(s,x(s)) ds\int_{0}^{t} \mu_2(u,s)\zeta_2(s,x(s)) ds\Big |\\
& + \lambda \Big|\int_{0}^{t} \mu_1(t,s)\zeta_1(s,x(s)) ds\int_{0}^{t} \mu_2(u,s)\zeta_2(s,x(s)) ds\\ & - \int_{0}^{t} \mu_1(u,s)\zeta_1(s,x(s)) ds \int_{0}^{t} \mu_2(u,s)\zeta_2(s,x(s)) ds\Big|\\ & +\lambda\Big| \int_{0}^{u} \mu_2(u,s)\zeta_2(s,x(s)) ds \Big\{\int_{0}^{t} \mu_1(u,s)\zeta_1(s,x(s)) ds -\int_{0}^{u} \mu_1(u,s)\zeta_1(s,x(s)) ds\Big\}\\ & +\int_{0}^{t} \mu_1(u,s)\zeta_1(s,x(s)) ds \int_{u}^{t} \mu_2(u,s)\zeta_2(s,x(s)) ds\Big|.
%&= |h(t,x(t))-h(t,x(u))|+\lambda \Big|\int_{0}^{t} \mu_1(t,s)\zeta_1(s,x(s)) ds \int_{0}^{t}\Big\{\mu_2(t,s)\zeta_2(s,x(s))-\mu_2(u,s)\zeta_2(s,x(s)) \Big\} ds \Big|\\&+\lambda \Big|\int_{0}^{t} \mu_2(u,s)\zeta_2(s,x(s)) ds \int_{0}^{t}\Big\{\mu_1(t,s)\zeta_1(s,x(s))-\mu_1(u,s)\zeta_1(s,x(s)) \Big\} ds \Big|\\&+\lambda\Big| \int_{0}^{u} \mu_1(u,s)\zeta_1(s,x(s)) ds \int_{u}^{t} \mu_2(u,s)\zeta_2(s,x(s))+\int_{0}^{u} \mu_2(u,s)\zeta_2(s,x(s)) ds \int_{u}^{t} \mu_2(u,s)\zeta_2(s,x(s)) ds\Big|.
%&= |h(t,x(t))-h(t,x(u))|+\lambda \Big|\int_{0}^{t} \mu_1(t,s)\zeta_1(s,x(s)) ds \int_{0}^{t} \mu_2(t,s)\zeta_2(s,x(s)) ds\\ & -\int_{0}^{t} \mu_1(t,s)\zeta_1(s,x(s)) ds\int_{0}^{t} \mu_2(u,s)\zeta_2(s,x(s)) ds\Big |  \\ &
% +\lambda \Big|\int_{0}^{t} \mu_1(t,s)\zeta_1(s,x(s)) ds\int_{0}^{t} \mu_2(u,s)\zeta_2(s,x(s)) ds\\&- \int_{0}^{t} \mu_1(u,s)\zeta_1(s,x(s)) ds \int_{0}^{t} \mu_2(u,s)\zeta_2(s,x(s)) ds\Big|\\ &  \lambda\Big| \int_{0}^{u} \mu_1(u,s)\zeta_1(s,x(s)) ds \int_{u}^{t} \mu_2(u,s)\zeta_2(s,x(s))+\int_{0}^{t} \mu_2(u,s)\zeta_2(s,x(s)) ds \int_{u}^{t} \mu_1(u,s)\zeta_1(s,x(s)) ds\Big|.
\end{align*}
Therefore, 
\begin{align}
&|Tx(t)-Tx(u)|\nonumber\\
& \leq \gamma  |x(t)-x(u)|+\lambda \Big|\int_{0}^{t} \mu_1(t,s)\zeta_1(s,x(s)) ds\Big|\nonumber\\ &\Big| \int_{0}^{t}\Big\{ \mu_2(t,s)\zeta_2(s,x(s))  - \mu_2(u,s)\zeta_2(s,x(s)) ds\Big\}\Big | \nonumber \\ & +\lambda \Big|\int_{0}^{t} \mu_2(u,s)\zeta_2(s,x(s)) ds\Big|\Big|\Big\{\int_{0}^{t} \mu_1(t,s)\zeta_1(s,x(s)) -  \mu_1(u,s)\zeta_1(s,x(s)) \Big\}ds \Big|\nonumber\\ & + \lambda\Big| \int_{0}^{u} \mu_1(u,s)\zeta_1(s,x(s)) ds \Big|\Big|\int_{u}^{t} \mu_2(u,s)\zeta_2(s,x(s))\Big|\nonumber\\ &+\lambda\Big|\int_{0}^{u} \mu_2(u,s)\zeta_2(s,x(s)) ds \Big|\Big|\int_{u}^{t} \mu_2(u,s)\zeta_2(s,x(s)) ds\Big|\nonumber\\
%
%&=\gamma |x(t)-x(u)|+\lambda \Big|\int_{0}^{t} \mu_1(t,s)\zeta_1(s,x(s)) ds\Big|\Big| \int_{0}^{t} \mu_2(t,s)\zeta_2(s,x(s)) ds\\ & -\int_{0}^{t} \mu_1(t,s)\zeta_1(s,x(s)) ds\int_{0}^{t} \mu_2(u,s)\zeta_2(s,x(s)) ds\Big |  \nonumber\\ & +\lambda \Big|\int_{0}^{t} \mu_1(t,s)\zeta_1(s,x(s)) ds\Big|\Big|\int_{0}^{t} \mu_2(u,s)\zeta_2(s,x(s)) ds- \int_{0}^{t} \mu_1(u,s)\zeta_1(s,x(s)) ds \int_{0}^{t} \mu_2(u,s)\zeta_2(s,x(s)) ds\Big|\nonumber\\ & \lambda\Big| \int_{0}^{u} \mu_1(u,s)\zeta_1(s,x(s)) ds \Big|\Big|\int_{u}^{t} \mu_2(u,s)\zeta_2(s,x(s))\nonumber\\ &+\lambda\Big|\int_{0}^{t} \mu_2(u,s)\zeta_2(s,x(s)) ds \Big|\Big|\int_{u}^{t} \mu_1(u,s)\zeta_1(s,x(s)) ds\Big|\nonumber\\
%%
&\leq \gamma |x(t)-x(u)| +\lambda A_1\Big| \int_{0}^{t} \Big\{\mu_2(t,s)\zeta_2(s,x(s))- \mu_2(u,s)\zeta_2(s,x(s))\Big\} ds\Big|\nonumber\\ & +\lambda \Big|\int_{0}^{u} \mu_2(u,s)\zeta_2(s,x(s)) ds\Big|\Big|\Big\{\int_{0}^{t} \mu_1(t,s)\zeta_1(s,x(s)) ds-  \mu_1(u,s)\zeta_1(s,x(s)) \Big\}ds \Big| \nonumber\\ & +\lambda \Big|\int_{u}^{t} \mu_2(u,s)\zeta_2(s,x(s)) ds\Big|\Big|\int_{0}^{t} \Big\{\mu_1(t,s)\zeta_1(s,x(s)) ds-  \mu_1(u,s)\zeta_1(s,x(s)) \Big\}ds \Big| \nonumber\\& +\lambda A_1 \Big|\int_{u}^{t} \mu_2(u,s)\zeta_2(s,x(s))\Big| + \lambda A_2\Big|\int_{u}^{t}\mu_2(u,s)\zeta_2(s,x(s))ds\Big|\nonumber\\
%\lambda\Big|\int_{0}^{u} \mu_2(u,s)\zeta_2(s,x(s)) ds \Big|\Big|\int_{u}^{t} \mu_1(u,s)\zeta_1(s,x(s)) ds\Big|\nonumber\\& +\lambda\Big|\int_{u}^{t} \mu_2(u,s)\zeta_2(s,x(s)) ds \Big|\Big|\int_{u}^{t} \mu_1(u,s)\zeta_1(s,x(s)) ds\Big|\nonumber\\
%
&\leq  \gamma |x(t)-x(u)|+\lambda A_1\Big| \int_{0}^{t} \Big\{\mu_2(t,s)\zeta_2(s,x(s))- \mu_2(u,s)\zeta_2(s,x(s))\Big\} ds\Big|\nonumber\\ & +\lambda A_2 \Big|\int_{0}^{t}\Big\{ \mu_1(t,s)\zeta_1(s,x(s)) ds-  \mu_1(u,s)\zeta_1(s,x(s))\Big\} ds \Big| \nonumber \\ & +\lambda \Big|\int_{u}^{t} \mu_2(u,s)\zeta_2(s,x(s)) ds \Big|\Big|\int_{0}^{t}\Big\{ \mu_1(t,s)\zeta_1(s,x(s)) ds-  \mu_1(u,s)\zeta_1(s,x(s))\Big\} ds \Big| \nonumber\\& +\lambda A_1 \Big|\int_{u}^{t} \mu_2(u,s)\zeta_2(s,x(s))\Big| +\lambda A_2 \Big|\Big|\int_{u}^{t} \mu_2(u,s)\zeta_2(s,x(s)) ds\Big|.\label{ttt1}
%& +\lambda\Big|\int_{u}^{t} \mu_2(u,s)\zeta_2(s,x(s)) ds \Big|\Big|\int_{u}^{t} \mu_1(u,s)\zeta_1(s,x(s)) ds\Big|. 
\end{align}
 Since $x(s)$ is continuous and  $s\in[0,L]$, we have $x(s)\in [a,b]$ for some $a,b \in \mathbb{R}$ with $a<b$ for all $s\in[0,L]$. Therefore, the functions $\mu_1:[0.L]\times [0,L] \to \mathbb{R}$ and $\zeta_1:[0.L]\times [a,b] \to \mathbb{R}$ are continuous and hence there exists a real number $B_1^L>0$ such that $|\mu_1(u,s)\zeta_1(s,x(s))|< B_1^L$ for all $u, s\in[0,L]$. Similarly there exists a real number $B_2^L>0$ such that $|\mu_2(u,s)\zeta_1(s,x(s))|< B_2^L$ for all $u, s\in[0,L]$. Next, we assume that 
\[w^L(\mu_1,\zeta_1,\varepsilon)= \sup \{|\mu_1(t,s)\zeta_1(s,x)-\mu_1(u,s)\zeta_1(s,x)|:u,t \in [0,L], |t-u|\leq \varepsilon\}\]
and
\[w^L(\mu_2,\zeta_2,\varepsilon)= \sup \{|\mu_2(t,s)\zeta_2(s,x)-\mu_2(u,s)\zeta_2(s,x)|:u,t \in [0,L], |t-u|\leq \varepsilon\}.\]
Using the above facts in \eqref{ttt1}, we get
\begin{align}
|Tx(t)-Tx(u)|&<\gamma |x(t)-x(u)|+\lambda L A_1  w^L(\mu_2,\zeta_2,\varepsilon)+\lambda L A_2  w^L(\mu_1,\zeta_1,\varepsilon) \nonumber \\ &+ \lambda L  B_2^L w^L(\mu_1,\zeta_1,\varepsilon) |t-u| +\lambda A_1 B_2^L |t-u|+\lambda A_2 B_2^L |t-u|. \label{eq3}
\end{align}
Let $\Lambda = \lambda L A_1  w^L(\mu_2,\zeta_2,\varepsilon)+\lambda L A_2  w^L(\mu_1,\zeta_1,\varepsilon)$ and $G(t,u)=\lambda L  B_2^L w^L(\mu_1,\zeta_1,\varepsilon) |t-u| +\lambda A_1 B_2^L |t-u|+\lambda A_2 B_1^L |t-u| $. Then from \eqref{eq3}, we get 
\begin{equation*} \label{eq4}
|Tx(t)-Tx(u)| <\gamma |x(t)-x(u)| +\Lambda+ G(t,u),
\end{equation*}
for all $t,u \in [0,L]$ and $|t-u|\leq \varepsilon$. So we have
\begin{align*}
&\sup\{|Tx(t)-Tx(u)|:t,u \in [0,L], |t-u|\leq \varepsilon\}\\&\leq \gamma \sup\{|x(t)-x(u)|:t,u \in [0,L], |t-u|\leq \varepsilon\}+\Lambda\\& +\sup\{G(t,u):t,u \in [0,L], |t-u|\leq \varepsilon\}\\
&\Rightarrow w^L(Tx,\varepsilon) \leq \gamma w^L(x,\varepsilon) +\Lambda+ \sup\{G(t,u):t,u \in [0,L], |t-u|\leq \varepsilon\}.\\
\end{align*}
The above relation holds for all $x\in A$. Therefore we have
\begin{align}
w^L(TA,\varepsilon) &\leq \gamma w^L(A,\varepsilon) +\Lambda+ \sup\{G(t,u):t,u \in [0,L], |t-u|\leq \varepsilon\}\nonumber\\
\Rightarrow \lim_{\varepsilon \to 0}w^L(TA,\varepsilon) &\leq \gamma \lim_{\varepsilon \to 0} w^L(A,\varepsilon) 
+\lim_{\varepsilon \to 0}\Lambda+ \lim_{\varepsilon \to 0}\sup\{G(t,u):t,u \in [0,L], |t-u|\leq \varepsilon\}\nonumber\\
\Rightarrow w_0^L(TA)&\leq \gamma w_0^L(A)\nonumber\\
\Rightarrow \lim_{L \to \infty}w_0^L(TA)&\leq \gamma \lim_{L \to \infty} w_0^L(A)\nonumber\\
\Rightarrow w_0(TA)&\leq \gamma w_0(A). \label{ee1}
\end{align}
%Therefore,
%\begin{align}
%O(f;w^L(TA,\varepsilon))&\leq O\left(f;\beta w^L(A,\varepsilon) +\Lambda+ \sup\{G(t,u):t,u \in [0,L], |t-u|\leq \varepsilon\}\right)\nonumber\\
%&\leq \beta O(f;w^L(A,\varepsilon))+O(f;\Lambda)+O\left(f;\sup\{G(t,u):t,u \in [0,L], |t-u|\leq \varepsilon\}\right)\nonumber\\
%\Rightarrow \lim_{\varepsilon\to 0} \mathcal{O}(f;w^L(TA,\varepsilon))&\leq \beta  \lim_{\varepsilon\to 0} \mathcal{O}(f;w^L(A,\varepsilon)) +\mathcal{O}(f;\lim_{\varepsilon\to 0} \Lambda) \nonumber\\ & +\mathcal{O}\left(f;\displaystyle \lim_{\varepsilon\to 0} \sup\{G(t,u):t,u \in [0,L], |t-u|\leq \varepsilon\}\right)\\
%\Rightarrow \lim_{\varepsilon\to 0} \mathcal{O}(f;w^L(TA,\varepsilon))&\leq \beta  \lim_{\varepsilon\to 0} \mathcal{O}(f;w^L(A,\varepsilon))\nonumber\\
%\Rightarrow \lim_{L \to \infty} \lim_{\varepsilon\to 0} \mathcal{O}(f;w^L(TA,\varepsilon))&\leq \beta \lim_{L \to \infty} \lim_{\varepsilon\to 0} \mathcal{O}(f;w^L(A,\varepsilon))\nonumber\\
%\Rightarrow \mathcal{O}(f;w_0(TA))&\leq \beta \mathcal{O}(f;w_0(A))\label{ee1}.
%\end{align}

Now for any $x,y \in A$, we have
\begin{align*}
&|Tx(t)-Ty(t)|\\
&=\Bigg|g(t,x(t))+\lambda \int_{0}^{t} \mu_1(t,s)\zeta_1(s,x(s)) ds \int_{0}^{t} \mu_2(t,s)\zeta_2(s,x(s)) ds - g(t,y(t))\\ &+\lambda \int_{0}^{t} \mu_1(t,s)\zeta_1(s,y(s)) ds \int_{0}^{t} \mu_2(t,s)\zeta_2(s,y(s)) ds \Bigg|\\
&\leq |g(t,x(t))-g(t,y(t))|+\lambda \Bigg|\int_{0}^{t} \mu_1(t,s)\zeta_1(s,x(s)) ds \int_{0}^{t} \mu_2(t,s)\zeta_2(s,x(s)) ds\\&-\int_{0}^{t} \mu_1(t,s)\zeta_1(s,y(s)) ds \int_{0}^{t} \mu_2(t,s)\zeta_2(s,y(s)) ds \Bigg|\\
&\leq \gamma |x(t)-y(t)|+\lambda\Bigg|\int_{0}^{t} \mu_1(t,s)\zeta_1(s,x(s)) ds -\int_{0}^{t} \mu_1(t,s)\zeta_1(s,y(s)) ds \Bigg|\\ & \Bigg|\int_{0}^{t} \mu_2(t,s)\zeta_2(s,x(s)) ds -\int_{0}^{t} \mu_2(t,s)\zeta_2(s,y(s)) ds\Bigg|\\&+\lambda\Bigg|\int_{0}^{t} \mu_1(t,s)\zeta_1(s,y(s)) ds\Bigg| \Bigg|\int_{0}^{t} \mu_2(t,s)\zeta_2(s,x(s)) ds -\int_{0}^{t} \mu_2(t,s)\zeta_2(s,y(s)) ds\Bigg|\\& +\lambda\Bigg|\int_{0}^{t} \mu_2(t,s)\zeta_2(s,y(s)) ds\Bigg| \Bigg|\int_{0}^{t} \mu_1(t,s)\zeta_1(s,x(s)) ds -\int_{0}^{t} \mu_1(t,s)\zeta_1(s,y(s)) ds\Bigg|\\
&\leq \gamma |x(t)-y(t)|+\lambda\Bigg|\int_{0}^{t} \mu_1(t,s)\zeta_1(s,x(s)) ds -\int_{0}^{t} \mu_1(t,s)\zeta_1(s,y(s)) ds \Bigg|\\ & \Bigg|\int_{0}^{t} \mu_2(t,s)\zeta_2(s,x(s)) ds -\int_{0}^{t} \mu_2(t,s)\zeta_2(s,y(s)) ds\Bigg|\\&+\lambda A_1\Bigg|\int_{0}^{t} \mu_2(t,s)\zeta_2(s,x(s)) ds -\int_{0}^{t} \mu_2(t,s)\zeta_2(s,y(s)) ds\Bigg| \\ &+\lambda A_2 \Bigg|\int_{0}^{t} \mu_1(t,s)\zeta_1(s,x(s)) ds -\int_{0}^{t} \mu_1(t,s)\zeta_1(s,y(s)) ds\Bigg|\\
&\leq \gamma diam~A(t) +\lambda\Bigg|\int_{0}^{t} \mu_1(t,s)\zeta_1(s,x(s)) ds -\int_{0}^{t} \mu_1(t,s)\zeta_1(s,y(s)) ds \Bigg|\\ & \Bigg|\int_{0}^{t} \mu_2(t,s)\zeta_2(s,x(s)) ds -\int_{0}^{t} \mu_2(t,s)\zeta_2(s,y(s)) ds\Bigg|\\&+\lambda A_1\Bigg|\int_{0}^{t} \mu_2(t,s)\zeta_2(s,x(s)) ds -\int_{0}^{t} \mu_2(t,s)\zeta_2(s,y(s)) ds\Bigg| \\ &+\lambda A_2 \Bigg|\int_{0}^{t} \mu_1(t,s)\zeta_1(s,x(s)) ds -\int_{0}^{t} \mu_1(t,s)\zeta_1(s,y(s)) ds\Bigg|.
\end{align*}
Therefore, 
\begin{align*}
diam ~TA(t)&\leq \gamma~~ diam~A(t)\nonumber\\& +\sup_{x,y\in A}\Bigg\{\lambda\Bigg|\int_{0}^{t} \mu_1(t,s)\zeta_1(s,x(s)) ds -\int_{0}^{t} \mu_1(t,s)\zeta_1(s,y(s)) ds \Bigg|\nonumber\\ & \Bigg|\int_{0}^{t} \mu_2(t,s)\zeta_2(s,x(s)) ds -\int_{0}^{t} \mu_2(t,s)\zeta_2(s,y(s)) ds\Bigg|\nonumber\\&+\lambda A_1\Bigg|\int_{0}^{t} \mu_2(t,s)\zeta_2(s,x(s)) ds -\int_{0}^{t} \mu_2(t,s)\zeta_2(s,y(s)) ds\Bigg| \nonumber\\ &+\lambda A_2 \Bigg|\int_{0}^{t} \mu_1(t,s)\zeta_1(s,x(s)) ds -\int_{0}^{t} \mu_1(t,s)\zeta_1(s,y(s)) ds\Bigg|\Bigg\}\nonumber\\
%\Rightarrow diam ~TA(t)&\leq \gamma \mathcal{O}(\xi;diam~A(t))+\mathcal{O}\Bigg(\xi;\sup_{x,y\in E}\Bigg\{\lambda\Bigg|\int_{0}^{t} \mu_1(t,s)\zeta_1(s,x(s)) ds\nonumber\\& -\int_{0}^{t} \mu_1(t,s)\zeta_1(s,y(s)) ds \Bigg| \Bigg|\int_{0}^{t} \mu_2(t,s)\zeta_2(s,x(s)) ds -\int_{0}^{t} \mu_2(t,s)\zeta_2(s,y(s)) ds\Bigg|\nonumber\\&+\lambda A_1\Bigg|\int_{0}^{t} \mu_2(t,s)\zeta_2(s,x(s)) ds -\int_{0}^{t} \mu_2(t,s)\zeta_2(s,y(s)) ds\Bigg| \nonumber\\ &+\lambda A_2 \Bigg|\int_{0}^{t} \mu_1(t,s)\zeta_1(s,x(s)) ds -\int_{0}^{t} \mu_1(t,s)\zeta_1(s,y(s)) ds\Bigg|\Bigg\}\Bigg).
\end{align*}
From the previous inequality, we obtain,
\begin{align}
&\limsup_{t\to \infty} diam ~TA(t)\nonumber\\&\leq \gamma \limsup_{t\to \infty} diam ~TA(t)\nonumber\\& + \lambda \limsup_{t\to \infty}\sup_{x,y\in A}\Bigg\{\Bigg|\int_{0}^{t} \mu_1(t,s)\zeta_1(s,x(s)) ds -\int_{0}^{t} \mu_1(t,s)\zeta_1(s,y(s)) ds \Bigg|\nonumber\\& \Bigg|\int_{0}^{t} \mu_2(t,s)\zeta_2(s,x(s)) ds -\int_{0}^{t} \mu_2(t,s)\zeta_2(s,y(s)) ds\Bigg|\Bigg\}\nonumber\\&+\lambda A_1 \limsup_{t\to \infty}\sup_{x,y\in A}\Bigg\{\Bigg|\int_{0}^{t} \mu_2(t,s)\zeta_2(s,x(s)) ds -\int_{0}^{t} \mu_2(t,s)\zeta_2(s,y(s)) ds\Bigg|\Bigg\}\nonumber \\ &+\lambda A_2 \limsup_{t\to \infty}\sup_{x,y\in A}\Bigg\{\Bigg|\int_{0}^{t} \mu_1(t,s)\zeta_1(s,x(s)) ds -\int_{0}^{t} \mu_1(t,s)\zeta_1(s,y(s)) ds\Bigg|\Bigg)\nonumber\\
\Rightarrow \alpha(TA)&\leq \gamma \alpha(A) ~~~[\mbox{using assumption} ~~(ii)]\label{ee2}.
\end{align}
Adding \eqref{ee1} and \eqref{ee2}, we get 
\begin{align}
w_0(TA)+\alpha(TA)&\leq \gamma\{w_0(A)+\alpha(A)\} \nonumber\\
\Rightarrow \sigma(TA) &\leq \gamma\sigma(A)\nonumber\\
\Rightarrow \mathcal{O}(\xi;\sigma(TA))&\leq \gamma \mathcal{O}(\xi;\sigma(A))\label{ee3}.
\end{align}
We choose two real numbers $\alpha_1$ and $\alpha_2$ such that $0<\alpha_1,\alpha_2<1$ and $\alpha_1 \alpha_2\geq \gamma$. Let us choose $\mathcal{F}(x,y)=\max\{x,y\}$, $\eta(t)=t$, $\beta(t)=\alpha_1(t)$, $\varphi(t)=\frac{t}{2}$ and $\alpha(t)=\alpha_2$ for all $t$. Then we have 
\begin{align*}
\eta(\mathcal{O}(\xi;\mathcal{F}(\sigma(TA);\varphi(\sigma(TA)))))=\mathcal{O}(\xi;\sigma(TA))
\end{align*} and 
\begin{align*}
\alpha(\mathcal{O}(\xi;\eta(\sigma(A))))\beta (\mathcal{O}(\xi;\mathcal{F}(\sigma(A),\varphi(\sigma(A))))=\alpha_1 \alpha_2 \mathcal{O}(\xi;\sigma(A)).
\end{align*}
Therefore, from \eqref{ee3}, we obtain 
\begin{align*}
\eta(\mathcal{O}(\xi;\mathcal{F}(\sigma(TA);\varphi(\sigma(TA))))) \leq \alpha(\mathcal{O}(\xi;\eta(\sigma(A))))\beta (\mathcal{O}(\xi;\mathcal{F}(\sigma(A),\varphi(\sigma(A)))).
\end{align*}
The above relation is true for all non-empty subsets  $A$ of $E$. So by Theorem \ref{th1}, $T$ has a fixed point in $E$. Hence the integral equation given by Equation \ref{eq1}, has a solution.
%\Bigg|\Bigg|
\end{proof}

\noindent{\bf Acknowledgements:}\\
The authors would like to thank Prof. D.R. Sahu for valuable comments during the preparation of this manuscript. The second named author would like to convey his cordial thanks to CSIR, New Delhi, India for their financial support.  %%%%%%%%%%%%%%%%%%%%%%%%%%%%%%%%

%\bibliography{Paper}
\bibliographystyle{plain}

\end{document}